\documentclass[11pt]{amsart}	

\usepackage{fullpage}

\usepackage{amssymb, amsmath, amsthm, color, amsfonts, caption, url, longtable, mathtools, hyperref}

\newtheorem{theorem}{Theorem}
\numberwithin{theorem}{section}
\newtheorem{corollary}[theorem]{Corollary}
\newtheorem{lemma}[theorem]{Lemma}

\newtheorem{remark}[theorem]{Remark}

\newcommand{\modgroup}{\operatorname{SL}_2(\mathbb Z)}
\newcommand\Tstrut{\rule{0pt}{2.6ex}}  
\newcommand\Bstrut{\rule[-0.9ex]{0pt}{0pt}}

\newcommand{\N}{\mathbb{N}}
\newcommand{\Z}{\mathbb{Z}}
\newcommand{\C}{\mathbb{C}}
\newcommand{\Q}{\mathbb{Q}}

\newcommand{\G}{\Gamma}

\newcommand{\SL}{\rm{SL}_2(\mathbb{Z})}
\newcommand{\leg}[2]{\big(\frac{#1}{#2}\big)}
\newcommand{\abcdmatrix}{\begin{psmallmatrix}a&b\\c&d\end{psmallmatrix}}

\newcommand{\spc}{\hspace{4pt} }

\newcommand{\rmhs}{\color{black}}

\begin{document}

\title{Eta-quotients of Prime or Semiprime Level and Elliptic Curves}

\author{Michael Allen, Nicholas Anderson, Asimina Hamakiotes, Ben Oltsik, Holly Swisher}




\thanks{This work was supported by NSF grant DMS-1359173.}

\maketitle

\begin{abstract}
From the Modularity Theorem proven by Wiles, Taylor, et al, we know that all elliptic curves are modular.  It has been shown by Martin and Ono exactly which are represented by eta-quotients, and some examples of elliptic curves represented by modular forms that are linear combinations of eta-quotients have been given by Pathakjee, RosnBrick, and Yoong.  

In this paper, we first show that eta-quotients which are modular for any congruence subgroup of level $N$ coprime to $6$ can be viewed as modular for $\Gamma_0(N)$.  We then categorize when even weight eta-quotients can exist in $M_k (\G_1(p))$ and $M_k (\G_1(pq))$, for distinct primes $p,q$.  We conclude by providing some new examples of elliptic curves whose corresponding modular forms can be written as a linear combination of eta-quotients, and describe an algorithmic method for finding additional examples. 
\end{abstract}


\section{Introduction and Statement of Results}\label{intro}

Dedekind's eta-function $\eta(\tau)$, defined for $\tau \in\mathbb{H}:=\{\tau \in \C \mid  \rm{Im}(\tau)>0\}$ by 
\[
\eta(\tau) := q^{\frac{1}{24}}\prod_{n=1}^\infty (1-q^n),
\]
where $q:=e^{2 \pi i\tau}$, is arguably the most well-known half-integral weight modular form.  Its modular transformation properties for a matrix $A = \left( \begin{smallmatrix}  a & b \\ c & d \end{smallmatrix} \right) \in \SL$ are given by
\begin{equation}
\eta(A\tau)=v(A)(c\tau+d)^{\frac{1}{2}}\eta(\tau),
\end{equation}
where
\begin{equation}
v(A):=
\begin{cases}
\leg{d}{|c|}e^{\frac{i\pi}{12}\big((a+d)c+bd(c^2-1)-3c\big)}& \text{if } c\equiv 1\pmod{2}\\ \leg{c}{d}e^{\frac{i\pi}{12}\big((a+d)c+bd(c^2-1)+3d-3-3cd\big)}&\text{if }d\equiv1\pmod{2}.
\end{cases}
\end{equation}

Dedekind's eta-function has been featured prominently in work motivated by Ramanujan's study of the integer partition counting function $p(n)$ (see \cite{Ono} and \cite{AO} for example), as well as in the representation theory of the Monster group, studied by Conway and Norton \cite{CN}, Borcherds \cite{Borcherds}, and others.  Due to its expression as a simple infinite product, it is easy to compute expansions numerically.  Thus it is useful when a modular form can be expressed in terms of products or quotients of $\eta(\tau)$. By \emph{eta-quotient of level $N$}, we mean a function of the form
\[
f(\tau) = \prod_{\delta\mid N}\eta^{r_\delta}(\delta\tau).
\]
Work on various classifications of eta-quotients has been of interest, see in particular work of 
Dummit, Kisilevsky, and McKay \cite{DKMK}, Martin \cite{Martin}, and Lemke Oliver \cite{LO}.  Moreover, the famous work of Wiles and Taylor  \cite{Wiles, TaylorWiles} in proving Fermat's Last Theorem, and in particular the Shimura-Taniyama conjecture, showed that elliptic curves were attached to modular forms, and thus the question of when these modular forms can be expressed in terms of eta-quotients is a natural one.  

\begin{theorem}[Modularity Theorem \cite{DS}]\label{modularityTheorem} Every elliptic curve $E$ over $\mathbb Q$ with conductor $N$ has an $L$-function 
\[
L(E, s) = \sum_{n=1}^\infty \frac{a_E(n)}{n^s}
\]
such that the Fourier series
\[
f(\tau) = \sum_{n=1}^\infty a_E(n)q^n,\hspace{5pt} q = e^{2\pi i \tau},\hspace{5pt} \tau \in \mathbb H
\]
represents a level $N$ cusp form of weight 2.
\end{theorem}

In light of Theorem \ref{modularityTheorem}, Martin and Ono \cite{MO} classified all eta-quotients which are weight $2$ newforms, as well as their associated elliptic curves.  It is natural to ask when elliptic curves are associated to modular forms that can be written as linear combinationa of eta-quotients.  Recently, Pathakjee, RosnBrick, and Toong \cite{PRY} demonstrated four such examples, utilizing spaces of cusp forms which are spanned by eta-quotients.  Further work on when spaces of modular forms are spanned by eta-quotients has been done by Rouse and Webb \cite{RW}, Arnold-Roksandich, James, and Keaton \cite{ARJK}, and Kilford \cite{Kilford}, for example.

Given a congruence subgroup group $\G \subseteq \SL$, we use the notation $S_k(\G)$, $M_k(\G)$, and $M_k^!(\G)$ to denote the complex vector spaces of weight $k$ cusp forms, holomorphic modular forms, and weakly holomorphic modular forms, respectively.  When $\G=\G_0(N)$ for a positive integer $N$, then we write $S_k(\G_0(N), \chi)$, $M_k(\G_0(N), \chi)$, and $M_k^!(\G_0(N), \chi)$ to denote the spaces of weight $k$ cusp forms, holomorphic modular forms, and weakly holomorphic modular forms, respectively, with Nebentypus $\chi$.

In this paper, we first show that eta-quotients which are modular for any congruence subgroup of level $N$ can be viewed as modular for $\Gamma_0(N)$.  We then categorize when even weight eta-quotients can exist in $M_k (\G_1(p))$ and $M_k (\G_1(pq))$, for distinct primes $p,q$.  We conclude by providing some new examples of elliptic curves whose corresponding modular forms can be written as a linear combination of eta-quotients, and describe an algorithmic method for finding additional examples. 

\subsection{Viewing eta-quotients over $\G_0(N)$.}
In order to state our first result, we review a few key theorems about the modularity of eta-quotients and their orders of vanishing at cusps.

The following well-known theorem originating in work of Newman \cite{Newman, Newman2} as well as Gordon and Hughes \cite{GH}, provides explicit modularity properties with respect to $\G_0(N)$ for eta-quotients of specific shapes.


\begin{theorem}[\cite{WebOf}, Theorem 1.64] \label{1.64}
Let $f$ be the eta-quotient of level $N$ given by
\[ \eta_g(\tau)=\prod_{\delta\mid N}\eta^{r_\delta}(\delta\tau).\]
If $\eta_g$ satisfies both
\begin{align*}
\sum_{\delta\mid N} \delta r_\delta&\equiv 0 \pmod{24}, \\
\sum_{\delta\mid N} \frac{N}{\delta}r_\delta&\equiv 0\pmod{24},
\end{align*}
then for $k=\frac{1}{2}\sum_{\delta\mid N}r_\delta $ and $\chi(d)=\leg{(-1)^ks}{d}$ where $s=\prod_{\delta\mid N}\delta^{r_\delta}$, we have $f \in M_k^!(\Gamma_0(N),\chi)$, i.e for all $M=\abcdmatrix\in \Gamma_0(N)$
\[f(M\tau)=\chi(d)(c\tau+d)^k f(\tau).\]
\end{theorem}

\begin{remark} \label{1.64reduct}
 In the case where $\gcd(N,6)=1,$ the two conditions above are equivalent.  This is because any $\delta \mid N$  must satisfy $\gcd(N,24)=1$, and hence $\delta$ is its own inverse modulo $24$.  Thus,
 \[
 \sum_{\delta \mid N} \frac{N}{\delta}r_\delta \equiv N \sum_{\delta \mid N} \delta r_\delta \pmod{24}.
 \]
 As $\gcd(N,24)=1$, $N$ is not a zero divisor in $\Z/24\Z,$ so 
 \[
 N\sum_{\delta \mid N} \delta r_\delta \equiv 0 \pmod{24} \spc \text{ if and only if } \spc \sum_{\delta \mid N} r_\delta \equiv 0 \pmod{24}.
 \]
\end{remark}

The next theorem originating in work of Ligozat \cite{Ligozat}, and further appearing in work of Biagioli \cite{Biagioli} and Martin \cite{Martin}, provides a mechanism for calculating the relative orders of vanishing at cusps for eta-quotients satisfying Theorem \ref{1.64}. 

Recall that if $h$ is the width of the cusp $r$ with respect to the group $\G$, then the \emph{order of vanishing relative to the group $\G$} of a modular form $f$ at the cusp $r$ is given by $v_{\G}\left(f, \frac{c}{d}\right) = h \cdot \rm{inv}_r(f)$, where $\rm{inv}_r(f)$ is the invariant order of vanishing of $f$ at $r$, and is always integral.  We note that for the cusp at infinity, we always have $v_{\G}\left(f, \frac{c}{d}\right) = \rm{inv}_r(f)$. 


\begin{theorem}[\cite{WebOf}, Theorem 1.65] \label{1.65}
Let $c$, $d$, and $N$ be positive integers with $d\mid N$ and $\gcd(c,d)=1$. Then if $f(\tau)$ is an eta-quotient satisfying the conditions given in Theorem \ref{1.64} for $N$, then the order of vanishing for $f(\tau)$ at the cusp $\frac{c}{d}$  relative to $\G_0(N)$ is
\[
v_{\frac{c}{d}} := v_{\G_0(N)}\left(f, \frac{c}{d}\right) = \frac{N}{24}\sum_{\delta\mid N}\frac{\gcd(d,\delta)^2r_\delta}{\gcd(d,\frac{N}{d})d\delta}.
\]
\end{theorem}

Calculating orders of vanishing of modular forms at cusps can be extremely useful, due to the following result known as Sturm's bound.

\begin{theorem}[Sturm's Bound \cite{MDG}] \label{sturmBound} Let $\Gamma$ be a congruence subgroup and $f \in M_k(\Gamma)$.  Let $r_1, \ldots, r_t$ be the $\Gamma$-inequivalent cusps of $\Gamma$.  If
\[
  \sum_{i=1}^t \rm{inv}_{r_i}(f) > \frac{k[\modgroup : \{ \pm I\}\Gamma]}{12},
\]
then $f = 0$.
\end{theorem}

\begin{remark}\label{SturmRmk}
We note that Theorem \ref{sturmBound} provides a direct way to check if two modular forms are equal by considering their $q$-expansions.  Namely, if $f,g\in M_k(\G)$ and their Fourier expansions at $i\infty$ agree to a power of $q$ past the bound in Theorem \ref{sturmBound}, then they must be equal.  
\end{remark}

Our first theorem shows that when $\gcd(N,6)=1$, any eta-quotient which is modular for a congruence subgroup of level $N$ is in fact modular for $\G_0(N)$ for some Nebentypus character $\chi$.

\begin{theorem}\label{Gamma0Thm}
Let $N$ be a positive integer with $\gcd(N,6)=1$ and suppose $f(\tau) = \prod_{\delta \mid N} \eta^{r_\delta}(\delta \tau) \in M_k^!(\Gamma(N))$.  Then $f(\tau) \in M_k^!(\Gamma_0(N),\chi)$, where $\chi$ is defined as in Theorem \ref{1.64}.  
\end{theorem}

In Section \ref{proof1}, Theorem \ref{Gamma0Thm} is proved by showing that $f(\tau) \in M_k^!(\Gamma(N))$ must satisfy the conditions of Theorem \ref{1.64}. Using that fact that $M_k^!(\Gamma_0(N),\chi) \subset M_k^!(\Gamma_1(N))$ as well as Remark \ref{1.64reduct}, we thus get the following immediate corollary to the proof of Theorem \ref{Gamma0Thm}.

\begin{corollary}\label{etathm}
Let $N$ be a positive integer with $\gcd(N,6)=1$.  Then $f(\tau) = \prod_{\delta \mid N} \eta^{r_\delta}(\delta\tau) \in M_k^!(\Gamma_1(N))$ if and only if $f(\tau)$ satisfies the conditions in Theorem \ref{1.64}.
\end{corollary}

From Corollary \ref{etathm}, we see that despite there being many more cusps of $\Gamma_1(N)$ than $\Gamma_0(N)$, we only need to consider the cusps of $\Gamma_0(N)$ when calculating orders of vanishing of eta-quotients in $M_k^!(\Gamma_1(N))$.  In work of Martin \cite{Martin}, a complete set of representatives for the cusps of $\Gamma_0(N)$ is given by

\[
C_{\G_0(N)}  = \left\{ \frac{a_c}{c} \in \Q : c \mid N, 1\leq a_c\leq N, \gcd(a_c,N)=1 \text{ and } a_c \equiv a_c'  \!\!\!\! \pmod{\gcd(c,N/c)} \text{ iff } a_c = a_c' \right\}.
\]

In the case where $N$ is squarefree, $\gcd(c,N/c) = 1$ for all $c$ which divide $N$. This gives the following complete set of representatives:

\begin{equation}\label{cuspreps}
C_{\G_0(N)} = \left\{\frac{1}{d} \in \Q: d \mid N \right\}.
\end{equation}


\subsection{Classifications}
Our next results deal with classifying when eta-quotients can exist in spaces of holomorphic modular forms of even weight and prime or semiprime level.  The following theorem uses techniques from work of Arnold-Roksandich, James, and Keaton \cite{ARJK}, and also addresses an error in \cite[Cor. 3.2]{ARJK}. 

\begin{theorem}\label{correctThm}
Let $p\geq 5$ be prime, set $h=\frac{1}{2}\gcd(p-1,24)$, and let $k$ be an even integer. Then there exists $f=\eta^{r_1}(\tau)\eta^{r_p}(p\tau) \in M_k\big( \Gamma_1(p)\big)$ if and only if both of the following conditions hold.
\begin{enumerate}
 \item $h \mid k$
 \item It is not the case that $p \ne 5$, $p \equiv 5 \pmod{24}$, and $k = 2$.
\end{enumerate}
\end{theorem}

We also prove a theorem of similar flavor for eta-quotients of semiprime level, which extends previous work of Arnold-Roksandich, James, and Keaton \cite[Theorem 5.8]{ARJK2}. 

\begin{theorem} \label{mastersemi}
Let $p, q \ge 5$ be distinct primes, $N = pq$, and $k$ an even integer.  Let $h = \frac{1}{2}\gcd\left(24,p-1,q-1\right)$.  Then there exists $f(\tau) = \prod_{\delta | N}\eta^{r_\delta}(\delta \tau) \in M_k(\Gamma_1(pq))$ if and only if both of the following conditions hold.
\begin{enumerate}
\item$h \mid k$ 
\item It is not the case that $(p, q) \!\!\! \mod{24} \in \{(1, 5),(5,1),(5, 5)\}$, $p, q \ne 5$, and $k = 2$.
\end{enumerate}
\end{theorem}

\subsection{Writing modular forms attached to elliptic curves in terms of eta-quotients}

Utilizing our work in Section \ref{section4}, we explore when $S_2(\Gamma_0(pq))$ has a basis consisting of eta-quotients.  This allows us to provide examples of elliptic curves attached via the Modularity Theorem (Thm. \ref{modularityTheorem}) to linear combinations of eta-quotients.  

\begin{theorem}\label{N=35}
Let $E$ be an elliptic curve with conductor 35.  Then $E$ is associated via the Modularity Theorem  to the modular form $f(\tau) \in S_2(\Gamma_0(35))$ given by 
\[
f(\tau) = \eta(\tau)^2\eta(35\tau)^2 + \eta(5\tau)^2\eta(7\tau)^2.
\]
\end{theorem}

We are also able to provide an example even when $S_2(\Gamma_0(pq))$ does not have a basis consisting of eta-quotients.
 
\begin{theorem}\label{N=55}
Let $E$ be an elliptic curve with conductor 55.  Then $E$ is associated via the Modularity Theorem  to the modular form $f(\tau) \in S_2(\Gamma_0(55))$ given by 
\[
f(\tau)= \sum_{i=1}^{40} c_i \frac{g_i(\tau)}{a(\tau)}, 
\]
where in Section \ref{ellcurves} the coefficients $c_i$ are given in Table \ref{coeffTab}, and the eta-quotients $g_i(\tau)/a(\tau)$ are given in Table \ref{etaTab}. 
\end{theorem}

\subsection{Outline of the rest of the paper}
The rest of the paper is devoted to proving our main theorems, and surrounding discussions. In Section \ref{proof1}, we prove Theorem \ref{Gamma0Thm}.  In Section \ref{primelevel}, we prove Theorem \ref{correctThm}, while in Section \ref{semiprimelevel}, we prove Theorem \ref{mastersemi}, discussing also squarefree level cases. Finally, in Section \ref{ellcurves}, we prove Theorems \ref{N=35} and \ref{N=55}.  

\section{eta-quotients on $\Gamma(N)$ and $\Gamma_0(N)$.}\label{proof1}
In this section we prove Theorem \ref{Gamma0Thm}.  We begin with a lemma.

\begin{lemma}\label{Gamma0Lem}
Fix a positive integer $N$ with $\gcd(N,6)=1$, and suppose $f(\tau)=\prod_{\delta\mid N}\eta^{r_{\delta}}(\delta\tau)$ has the property that $f \in M_{k_f}^!(\Gamma(N))$, where $k_f= \frac12\sum_{\delta\mid N} r_\delta$, and
\begin{equation}
\sum_{\delta\mid N}\delta r_{\delta}\equiv t\pmod{24}.
\end{equation}
Then if $g(\tau)=\prod_{\delta\mid N}\eta^{\rho_\delta}(\delta\tau)$ is any eta-quotient of level $N$ satisfying
\begin{equation}
\sum_{\delta\mid N}\delta \rho_\delta\equiv -t\pmod{24},
\end{equation}
we must have that $g\in M_{k_g}^!(\Gamma(N))$, where $k_g=\frac{1}{2}\sum_{\delta\mid N}\rho_\delta $.
\end{lemma}

\begin{proof}
Observe that
\[
fg(\tau) = \prod_{\delta\mid N}\eta^{r_{\delta}+\rho_\delta}(\delta\tau),
\]
and by our hypotheses,
\[
\sum_{\delta\mid N}\delta (r_\delta+\rho_\delta)\equiv t-t\equiv 0\pmod{24}.
\]
Thus utilizing Remark \ref{1.64reduct}, we see that $fg(\tau)$ satisfies the conditions in Theorem \ref{1.64}. In particular we have that $fg \in M^!_k(\Gamma_0(N),\chi)$, where 
$k=\frac{1}{2}\sum_{\delta\mid N}\left(r_{\delta}+\rho_\delta \right)$, 
$\spc \chi(d)=\left(\frac{(-1)^{k}s}{d}\right)$, and $\spc s=\prod_{\delta\mid N}\delta^{r_\delta+\rho_\delta}$.

Note that when $d \equiv 1 \pmod{N}$, we have $\chi(d)=1$ since $\chi$ is a Dirichlet character modulo $N$.  Thus we see that  for every $A=\abcdmatrix\in \Gamma(N)$, $g=\frac{fg}{f}$ satisfies
\begin{equation}
g(A\tau)=\frac{(fg)(A\tau)}{f(A\tau)}=\frac{\chi(d)(c\tau+d)^{k}g(\tau)f(\tau)}{(c\tau+d)^{k_f}f(\tau)}=(c\tau+d)^{k-k_f}g(\tau).
\end{equation}
In particular, $g\in M_{k_g}^!(\Gamma(N))$. 
\end{proof}

We are now able to prove Theorem \ref{Gamma0Thm}.

\begin{proof}[Proof of Theorem \ref{Gamma0Thm}]
Since we are assuming that $\gcd(N,6)=1$, it suffices to show that any eta-quotient $f(\tau) =\prod_{\delta\mid N}\eta^{r_{\delta}}(\delta\tau) \in M_{k_f}^!(\Gamma(N))$ satisfies the first condition in Theorem \ref{1.64}, namely that $\sum_{\delta\mid N}\delta r_{\delta}\equiv 0 \pmod{24}$.  

The contrapositive to Lemma \ref{Gamma0Lem} states that if for some residue class $t$ modulo $24$, there exists an eta-quotient $g(\tau) = \prod_{\delta \mid N} \eta^{\rho_\delta}(\delta \tau) \not\in M_{k_g}^!(\Gamma(N))$ such that
\begin{equation}
  \sum_{\delta \mid N} \delta \rho_\delta \equiv t \pmod{24}, \label{Gamma0cond}
 \end{equation}
then for any $f(\tau) =\prod_{\delta\mid N}\eta^{r_{\delta}}(\delta\tau) \in M_{k_f}^!(\Gamma(N))$, we must have that $\sum_{\delta\mid N}\delta r_{\delta}\not\equiv -t\pmod{24}$.

Hence, if we can find such a $g(\tau)$ for each nonzero residue class $t$ modulo $24$, we are left to conclude that any $f(\tau) =\prod_{\delta\mid N}\eta^{r_{\delta}}(\delta\tau) \in M_{k_f}^!(\Gamma(N))$ must satisfy $\sum_{\delta\mid N}\delta r_{\delta}\equiv 0\pmod{24}$, as desired.

Fix $1\leq t \leq 23$, and consider $g_t(\tau) = \eta^t(\tau)$.  Certainly $g_t$ satisfies \eqref{Gamma0cond}, so we need only show that $g_t$ is not a weakly holomorphic modular form for any $\Gamma(N)$ with $\gcd(N,6)=1$.  Fix such an $N$, and let $A=\abcdmatrix\in \SL$.  The transformation properties of $\eta(\tau)$ are given by (see Knopp \cite{knopp}) 
\[
\eta(A\tau)=v(A)(c\tau+d)^{\frac{1}{2}}\eta(\tau),
\]
where
\[ 
v(A):=
\begin{cases}
\leg{d}{|c|}e^{\frac{i\pi}{12}\big((a+d)c+bd(c^2-1)-3c\big)}& \text{if } c\equiv 1\pmod{2}\\ \leg{c}{d}e^{\frac{i\pi}{12}\big((a+d)c+bd(c^2-1)+3d-3-3cd\big)}&\text{if }d\equiv1\pmod{2}.\end{cases}
\]
One way to show that $\eta^t(\tau)\not\in M_{t/2}^!(\Gamma(N))$, is to find a matrix $A =\abcdmatrix \in \Gamma(N)$ such that $v(A)^t \ne 1$.  Consider the matrix $\begin{psmallmatrix}1&0\\N&1\end{psmallmatrix} \in \Gamma(N)$.  Since $N$ is odd, $v(A)$ evaluates in the case $c\equiv 1\pmod{2}$ above.  Since $d=1$, the generalized Legendre symbol in this expression is trivial.  So we get that 
\[
v(A)^t = e^{\frac{-tN\pi i}{12}}.
\]
Thus $v(A)^t \ne 1$ if and only if $-tN$ is not divisible by $24$.  But this is certainly true by our choice of $t$ and the fact that $\gcd(N,6)=1$. 
\end{proof}

\section{Eta-quotients of Prime Level}\label{primelevel}
In this section our goal is to prove Theorem \ref{correctThm}, which classifies when eta-quotients can exist in $M_k^!(\Gamma_1(p))$ for $p$ prime.  And of course by Theorem \ref{Gamma0Thm} we know that eta-quotients in $M_k^!(\Gamma_1(p))$ are actually in $M_k^!(\Gamma_0(p), \chi)$ for some character $\chi$.  By \eqref{cuspreps}, there are two cusps of $\Gamma_0(p)$: $1$ and $1/p$.  We first state the following useful result from a recent paper by Arnold-Roksandich, James and Keaton \cite{ARJK}, but offer an alternative proof which we will see  in Section 4 generalizes to squarefree levels.

\begin{lemma}[\cite{ARJK}, Theorem 1.2]\label{hLemma}
Fix a prime $p\geq 5$ and let $h=\frac{1}{2}\gcd(p-1,24)$.  There exists $f(\tau)=\eta^{r_1}(\tau)\eta^{r_p}(p\tau)$ in $M_k^!\big(\Gamma_1(p)\big)$ if and only if $h\mid k.$
\end{lemma}

\begin{proof}
Suppose there exists $f(\tau)=\eta^{r_1}(\tau)\eta^{r_p}(p\tau)$ in $M_k^!\big(\Gamma_1(p)\big)$.  By Corollary \ref{etathm} we see that $f$ satisfies the conditions in Theorem \ref{1.65}, and thus  
\begin{equation}\label{primev1}
  v_1 = \frac1{24}(pr_1 + r_p).
\end{equation}
Since $k=\frac{1}{2}(r_1 + r_p)$, \eqref{primev1} is equivalent to 
 
\begin{equation}
 \label{dioeqn1}
 24v_1 - (p-1)r_1 = 2k.
\end{equation}

Equation \eqref{dioeqn1} can be viewed as a linear Diophantine equation in variables $v_1$ and $r_1$.  In this light, the existence of a solution implies $\gcd(p-1,24) \mid 2k$.  As $p \ne 2$, $p-1$ is even and so $2 \mid \gcd(24,p-1)$.  Therefore, $h = \frac{1}{2} \gcd(p-1,24)$ is an integer and $h \mid k$.

Conversely, if $h \mid k$, then there exists integers $v_1$, $r_1$ which give a solution to \eqref{dioeqn1}.  Plugging these into \eqref{primev1} gives an integer $r_p$.  From \eqref{primev1} we see the first condition of Theorem \ref{1.64} is satisfied, and so since $(p,6)=1$, Remark \ref{1.64reduct} implies that $f(\tau) \in M_k^!(\Gamma_1(p))$.
\end{proof}

Our proof of Theorem \ref{correctThm} will also rely on the following result of Rouse and Webb \cite{RW}.

\begin{theorem}[\cite{RW}, Theorem 2] \label{expbound}
Suppose that $f(\tau) = \prod_{\delta|N}\eta(\delta \tau)^{r_\delta} \in M_k(\Gamma_0(N))$.  Then we have
\[
 \sum_{\delta|N}|r_\delta| \le 2k \prod_{p | N} \left(\frac{p+1}{p-1}\right)^{\min\{2, \operatorname{ord}_p(N)\}}.
\]
\end{theorem}

\begin{remark}\label{RWrmk}
The proof of this theorem carries through for spaces $M_k(\Gamma_0(N), \chi)$ with character.  Thus in light of Corollary \ref{etathm}, Theorem \ref{expbound} holds for any eta-quotient in $M_k(\Gamma_1(N))$. 
\end{remark}


We also require two additional lemmas which will allow us to deal with certain cases in the proof of Theorem \ref{correctThm}.  These lemmas both deal with the value $k(p+1)/12$, where $p\geq 5$ is prime, and $k$ is an even integer such that $h=\frac{1}{2}\gcd(p-1,24) \mid k$.  Note that in this case $k(p+1)/12$ must be an integer, since if $3$ doesn't divide $p+1$, then $3$ must divide $p-1$, and so $3$ divides $h$ and thus $k$.  

\begin{lemma}\label{evenexistence}
Fix a prime $p\geq 5$, set $h=\frac{1}{2}\gcd(p-1,24)$, and let $k$ be a positive even integer.  If ${k(p+1)/12}$ is even, then there exists an eta-quotient in $M_k(\Gamma_1(p))$.
\end{lemma}

\begin{proof}
Consider the eta-quotient
 \[
 f(\tau) = \eta^k(\tau)\eta^k(p\tau).
 \]
Since $k(p+1)/12$ is even, and by Remark \ref{1.64reduct}, we see that $f(\tau)$ satisfies the conditions of Theorem \ref{1.64}, and so $f(\tau) \in M_k^!(\Gamma_1(p))$.  Moreover, by Theorem \ref{1.65}, 
 \begin{equation}
  v_1 = v_{1/p} = \frac{k(p+1)}{24},
 \end{equation}
which is a positive integer. Thus $f(\tau) \in M_k(\Gamma_1(p))$.
\end{proof}

\begin{remark}
 Note that we did not need the assumption that $h \mid k$ in the previous lemma.  This is because when $k(p+1)/12$ is even, we must have $h \mid k$.  To see this, observe that if $k(p+1)/12 = 2n$, then $2k = 24n - k(p-1)$, and so $\gcd(24,p-1) \mid 2k$.
\end{remark}

The last lemma is a simple divisibility argument.

\begin{lemma}\label{evenlemma}
Let $p$ be prime, $h=\frac{1}{2}\gcd(p-1,24)$, and $k$ be an even integer. If $h\mid k$ and $k(p+1)/12$ is odd, then $(p-1)/2h$ is odd. 
\end{lemma}

\begin{proof}
We first note that when $p=2$, $(p-1)/2h = 1$ so is odd regardless of $k$.  We now let $p \geq 3$.  Suppose $(p-1)/2h$ is even.  By definition of $h$, $(p-1)/2h$ and $12/h$ are relatively prime, so $12/h$ must be odd.   But then $12$ is an odd integer times $h$, so we have that $4 \mid h$, and so $4\mid k$.  Thus $8\mid k(p+1)$, since $p$ is odd.  Therefore $k(p+1)/12$ is even, which contradicts our assumptions.  
\end{proof}

We are now able to prove Theorem \ref{correctThm}.

\begin{proof}[Proof of Theorem  \ref{correctThm}]

First, we assume that $f(\tau) = \eta^{r_1}(\tau)\eta^{r_p}(p\tau)\in M_k(\Gamma_1(p))$.  By Lemma \ref{hLemma}, we have that $h \mid k$.  To complete this direction of the proof, we show that there are no eta-quotients in $M_2(\Gamma_1(p))$ when $p \equiv 5$ (mod $24$) and $p > 5$.  To do this, we recall Remark \ref{RWrmk}, and employ Theorem \ref{expbound}.  Fix $p \equiv 5$ (mod $24$) with $p > 5$, and suppose $f(\tau) = \eta^{r_1}(\tau)\eta^{r_p}(p\tau)\in M_2(\Gamma_1(p))$.  Theorem \ref{expbound} gives us that
\[
 |r_1|+|r_p| \leq 4\left(\frac{p+1}{p-1}\right).
\]
This upper bound decreases as $p$ increases, so will be largest when $p = 29$, which gives a bound less than 5. Since $r_1$ and $r_p$ are integers, we have
\begin{equation}
 \label{primelevelRW2}
 |r_1|+|r_p| \leq 4.
\end{equation}
Moreover, by Theorem \ref{1.65} we have
\begin{align*}
 24v_1 &= r_1+pr_p \\
 24v_{1/p}  &= pr_1+r_p. 
\end{align*}
Since $f(\tau) \in M_k(\Gamma_1(p))$, it must be that $v_1,v_{1/p}\geq 0$.  This guarantees that $r_1,r_p \geq 0$.  Namely, if $r_1<0$, then using \eqref{primelevelRW2}, we see that
\[
 24v_{1/p} \leq -29|r_1| + |r_p| <0,
\]
which contradicts that $v_{1/p}\geq 0$.  Similarly, $r_p<0$ implies that $v_1 < 0$.  

Since $p \equiv 5$ (mod $24$), Corollary \ref{etathm} implies that
\begin{equation}
 \label{prime1.64failure}
 r_1+5r_p \equiv 0 \pmod{24}.
\end{equation}
However, there is no pair of non-negative integers $r_1,r_p$, not both zero, that satisfy both (\ref{primelevelRW2}) and (\ref{prime1.64failure}).  Thus we have a contradiction and so no such eta-quotient $f(\tau)$ can exist.


We now prove the converse by construction.  Suppose that $p\geq 5$ is a prime and $k$ an even integer such that $h=\frac{1}{2}\gcd(p-1,24) \mid k$.  We construct an eta-quotient in $M_k(\G_1(p))$ for every such case of $p$ and $k$ except when $k = 2$ and $p \equiv 5 \pmod{24}$, with $p > 5$.

We first consider the case where $p \not\equiv 5$ (mod $8$).  By Lemma \ref{evenexistence}, it suffices to prove that $k(p+1)/12$ is even, which since $k(p+1)/12$ is an integer, means showing that $8\mid k(p+1)$.  We already know that both $k$ and $p+1$ are even, so it suffices to show that either $k$ or $p+1$ is divisible by $4$.  If $p \equiv 3 \pmod 4$, then $4 \mid p+1$ so we are done.  If 
not, then $p \equiv 1 \pmod{8}$ and we have that $4 \mid h$ and so $4\mid k$.

We now consider the case when $p \equiv 5 \pmod{8}$.  In this case either $p \equiv 13 \pmod{24}$ or $p\equiv 5 \pmod{24}$.  We will show the existence of an eta-quotient in $M_h(\Gamma_1(p))$, since if $f(\tau)\in M_h(\Gamma_1(p))$, then $f^{k/h}(\tau) \in M_k(\Gamma_1(p))$.  Suppose that $p \equiv 13$ (mod $24$).  Then $h = 6$, and using Theorems \ref{1.64} and \ref{1.65} one can quickly check that
\[
\eta^9(\tau)\eta^3(p\tau) \in M_6(\Gamma_1(p)).
\]
Next, when $p = 5$ we have $h = 2$, and so
\[
\frac{\eta^5(p\tau)}{\eta(\tau)} \in M_2(\Gamma_1(p)).
\]
Finally, suppose $p \equiv 5$ (mod $24$), with $p \ne 5$ and $k \neq 2$.  It is sufficient to show that there exist eta-quotients $f_4(\tau) \in M_4(\Gamma_1(p))$ and $f_6(\tau) \in M_6(\Gamma_1(p))$ since either $4 \mid k$ or $4 \mid (k-6)$ and thus either $f_4^{k/4}\in M_k(\Gamma_1(p))$ or $f_4^{(k-6)/4}f_6\in M_k(\Gamma_1(p))$. We check using Theorems \ref{1.64} and \ref{1.65} that
\begin{align*}
 f_4(\tau) &:= \eta^4(\tau)\eta^4(p\tau) \in M_4(\Gamma_1(p))\\
 \intertext{and}
 f_6(\tau) &:= \eta^9(\tau)\eta^3(p\tau) \in M_6(\Gamma_1(p)),
\end{align*}
which resolves the final case. 


\end{proof}

\section{Eta-quotients of Semiprime and Squarefree Level}\label{semiprimelevel}\label{section4}
In this section, we generalize the results of Section 3 to semiprime level, and further to squarefree levels where possible.  We begin with a squarefree generalization of Lemma \ref{hLemma} which extends a result of Arnold-Roksandich, James, and Keaton \cite[Thoerem 5.8]{ARJK}.

\begin{theorem}\label{sfhlemma}
Let $N=p_1\cdots p_t$, a product of distinct primes with $p_i \geq 5$.  Define $h = \frac{1}{2} \gcd(p_1-1, \ldots, p_t-1, 24)$. There exists $f(\tau) = \prod_{\delta \mid N} \eta^{r_\delta}(\delta \tau) \in M_k^!(\Gamma_1(N))$ if and only if $h \mid k$.
\end{theorem}

\begin{proof}
Suppose there exists $f(\tau) = \prod_{\delta \mid N} \eta^{r_\delta}(\delta \tau) \in M_k^!(\Gamma_1(N))$.  By Corollary \ref{etathm} we see that $f$ satisfies the conditions in Theorem \ref{1.65}, and thus we can compute 
\begin{equation}
v_{1/N} =  \frac{1}{24} \sum_{\delta \mid N} \delta r_\delta. \label{squarefreevn}
\end{equation}
Since $k=\frac{1}{2}\sum_{\delta \mid N} r_\delta$, \eqref{squarefreevn} is equivalent to 
 
\begin{equation} \label{squarefreedio}
2k = 24v_{1/N} - \sum_{\delta \mid N} (\delta-1)r_\delta.
\end{equation}

Equation \eqref{squarefreedio} can be viewed as a linear Diophantine equation in the variables $v_{1/N}$, and $r_\delta$ for each $\delta \mid N$ with $\delta\neq 1$.  Thus, the existence of an integer solution to \eqref{squarefreedio} implies that
\begin{equation}\label{sqfreesoln}
\gcd_{\delta\mid N}(\delta-1,24) \mid 2k.
\end{equation}
However, we note that 
\begin{equation}\label{sqfreeh}
\gcd_{\delta\mid N}(\delta-1,24) = \gcd(p_1-1, \ldots, p_t-1, 24) = 2h,
\end{equation}
which follows from the fact that $ab-1=  (a-1)(b-1)+(a-1)+(b-1)$, so if $d\mid a-1$ and $d\mid b-1$, then $d\mid (ab-1)$.   As each $p_i$  is odd,  we have that $2 \mid \gcd(p_1-1, \ldots, p_t-1, 24)$.  Therefore, $h$ is an integer and we see from \eqref{sqfreesoln} and \eqref{sqfreeh} that $h \mid k$ as desired.

Conversely, if $h \mid k$, then there exists integers $v_{1/N}$, and $r_\delta$ for each $\delta \mid N$ with $\delta\neq 1$, which give a solution to \eqref{squarefreedio}.  Additionally defining
\[
r_1 = 2k - \sum_{\substack{\delta \mid N \\ \delta \ne 1}} r_\delta,
\]
gives that $k=\frac{1}{2}\sum_{\delta \mid N} r_\delta$.  Thus from \eqref{squarefreedio} we see that 
\[
\sum_{\delta \mid N} \delta r_\delta = 24 v_{1/N} \equiv 0 \pmod{24},
\]
and so by Remark \ref{1.64reduct}, $f(\tau) := \prod_{\delta \mid N} \eta^{r_\delta}(\delta\tau) \in M_k^!(\Gamma_1(N))$.
\end{proof}
 
 The next theorem computes the sum of the orders of vanishing of an eta-quotient in $M_k^!(\Gamma_1(N))$ in terms of the divisor function
 \[
 \sigma_a(n) := \sum_{d\mid n} d^a.
 \]
 
 \begin{theorem}\label{sfvansum}
Let $N=p_1\cdots p_t$, a product of distinct primes with $p_i \geq 5$, and let $f(\tau) = \prod_{\delta \mid N} \eta^{r_\delta}(\delta \tau) \in M_k^!(\Gamma_1(N))$.  Then
\[
\sum_{d \mid N} v_{1/d} = \frac{k\sigma_1(N)}{12}.
\]
\end{theorem}
 
\begin{proof}
Since $N$ is squarefree, Theorem \ref{1.65} gives that
\begin{equation} \label{vanishsum}
\sum_{d \mid N} v_{1/d} = \sum_{d \mid N} \frac{N}{24} \sum_{\delta \mid N} \frac{\gcd(d,\delta)^2 r_\delta}{d\delta} = \frac{1}{24} \sum_{\delta \mid N} r_\delta \sum_{d \mid N} \frac{N\gcd(d,\delta)^2}{d\delta}. 
\end{equation}
Since $\sum_{\delta \mid N} r_\delta = 2k$, it suffices to show that the inner sum in \eqref{vanishsum} is $\sigma_1(N)$.  We thus fix $\delta \mid N$, and aim to show that for each divisor $d'$ of $N$, there is a unique divisor  $d \mid N$ such that
\begin{equation} \label{goaleqn}
\frac{N\gcd(d,\delta)^2}{d\delta} = d'.
\end{equation}

We first show existence.  Given $d' \mid N$, set $d = \gcd(d',\delta)\gcd(N/d',N/\delta)$, which clearly divides $N$.  Since $\delta$, $d'$, and $d$ are all squarefree, we observe that $\gcd(d,\delta) = \gcd(d',\delta)$, since any prime divisor of $d$ and $\delta$ must divide $\gcd(d',\delta)$.

Our goal is to show \eqref{goaleqn}, which we rewrite as
\begin{equation} \label{goaleqn2}
\frac{N}{\gcd(N/d',N/\delta)} = \frac{\delta d'}{\gcd(d',\delta)}.
\end{equation}
The left hand side of \eqref{goaleqn2} is simply the product of the prime divisors of $N$ which are also prime divisors of $\delta$ or $d'$.  But this is also the right hand side of \eqref{goaleqn2} so we are done, and have shown existence. 

To determine uniqueness, we suppose $d_1, d_2 \mid N$ such that $\frac{N\gcd(d_1,\delta)^2}{d_1\delta}=\frac{N\gcd(d_2,\delta)^2}{d_2\delta}$, namely that $d_2\gcd(d_1,\delta)^2=d_1\gcd(d_2,\delta)^2$.  Since $N$ is squarefree, comparing the prime factorizations of $d_1$, $d_2$, $\delta$, and considering cases yields that $d_1=d_2$ as desired. 

We are now finished since we have shown that 
\begin{equation}\label{willuse}
\sum_{d \mid N} \frac{N\gcd(d,\delta)^2}{d\delta} = \sigma_1(N).
\end{equation}
\end{proof}
 
\begin{remark}
When $N$ is squarefree, $\sigma_1(N) = \prod_{p \mid N}(p+1)$.  Thus when $N=p_1\cdots p_t$, a product of distinct primes with $p_i \geq 5$, and $f(\tau) = \prod_{\delta \mid N} \eta^{r_\delta}(\delta \tau) \in M_k^!(\Gamma_1(N))$, we have by Theorem \ref{sfvansum} that
\[
\sum_{d \mid N}v_{1/d} = \frac{k\sigma_1(N)}{12} = \frac{k\prod_{p \mid N}(p+1)}{12}. 
\]
From this we can observe that $k\sigma_1(N)/12$ must always be an integer whenever $k$ is an even integer such that $h\mid k$, for $h = \frac{1}{2} \gcd(p_1-1, \ldots, p_t-1, 24)$. This is because if $3$ doesn't divide $p+1$ for any $p\mid N$, then $3$ must divide $p-1$ for all $p\mid N$, and so $3$ divides $h$ and thus $k$. 
\end{remark}
 
 We next generalize Lemma \ref{evenexistence} to squarefree levels.
 
 \begin{lemma}\label{sfexistence}
Let $N=p_1\cdots p_t$, a product of distinct primes with $p_i \geq 5$, and define $h = \frac{1}{2} \gcd(p_1-1, \ldots, p_t-1, 24)$.  Suppose $k$ is a positive integer such that $h\mid k$, $2^{t-1} \mid k$, and $2^t \mid \left(k\sigma_1(N)/12\right)$.  Then there exists $f(\tau) = \prod_{\delta \mid N} \eta^{r_\delta}(\delta\tau) \in M_k(\Gamma_1(N))$.
 \end{lemma}
 
 \begin{proof}
  Define $f(\tau)$ by
  \[f(\tau) = \prod_{\delta \mid N} \eta^{\frac{k}{2^{t-1}}}(\delta\tau).\]
  As $N$ is squarefree, there are $2^t$ divisors of $N$, and so $\sum_{\delta \mid N} \frac{k}{2^{t-1}} = 2k$.  Additionally, by our hypothesis $2^t \mid (k \sigma_1(N)/12)$, so we have
\[
\sum_{\delta \mid N} \delta \left(\frac{k}{2^{t-1}}\right) = \left(\frac{k}{2^{t-1}}\right) \sigma_1(N) = \frac{k \sigma_1(N) }{2^{t-1}} \equiv 0 \pmod{24}.
\]
Thus, $f(\tau)$ satisfies Theorem \ref{1.64}, and so by Corollary \ref{etathm}, $f(\tau) \in M_k^!(\Gamma_1(N))$.  Recalling \eqref{willuse}, we have by Theorem \ref{1.65} that
\[
v_{1/d} = \frac{k}{2^{t-1}} \cdot \frac{\sigma_1(N)}{24} = \frac{k\sigma_1(N)}{2^t\cdot 12},
\]
which is a positive integer by our hypotheses. Thus $f(\tau) \in M_k(\Gamma_1(N))$.
\end{proof}


We are now ready to prove Theorem \ref{mastersemi}.

\begin{proof}[Proof of Theorem \ref{mastersemi}]

First, we assume that 
$$
f(\tau) = \eta^{r_1}(\tau)\eta^{r_p}(p\tau)\eta^{r_q}(q\tau)\eta^{r_N}(N\tau)\in M_k(\Gamma_1(pq)),
$$ 
where $N=pq$ is a product of distinct primes $p,q\geq 5$, and $k$ is an even integer.  By Theorem \ref{sfhlemma}, we have that $h \mid k$.  To complete this forward direction of the proof, we show that there are no eta-quotients in $M_2(\Gamma_1(pq))$ when $(p,q) \equiv (1,5), (5,1), (5,5) \pmod{24}$ and $p,q > 5$.  Without loss of generality, we may assume $p$ and $q$ are chosen so that $p$ has a lesser or equal residue modulo $24$.

Recalling Remark \ref{RWrmk}, we use Theorem \ref{expbound}.  Fix $(p,q) \equiv (1,5), (5,5) \pmod{24}$ with $p,q > 5$, and suppose $f(\tau) = \eta^{r_1}(\tau)\eta^{r_p}(p\tau)\eta^{r_q}(q\tau)\eta^{r_N}(N\tau) \in M_2(\Gamma_1(pq))$.  Then by Theorem \ref{expbound},
\begin{equation*}
 |r_1|+|r_p| +  |r_q|+|r_N| \leq 4\left(\frac{p+1}{p-1}\right)\left(\frac{q+1}{q-1}\right).
\end{equation*}
This upper bound decreases as $p$ and $q$ increases, so the bound will be largest when $p=29$ and $q=53$. This gives a bound less than 5, so since the $r_\delta$ are integers,
\begin{equation}\label{spexpineq}
 |r_1|+|r_p| +  |r_q|+|r_N|  \leq 4.
\end{equation}

Moreover, by Theorem \ref{1.65},
\begin{align*}
  24v_1 &= Nr_1+qr_p+pr_q+r_N \\
  24v_{1/p} &= qr_1 + Nr_p + r_q + pr_N \\
  24v_{1/q} &= pr_1 + r_p + Nr_q + qr_N \\
  24v_{1/N} &= r_1 + pr_p + qr_q + Nr_N.
\end{align*}
Since $f(\tau) \in M_k(\Gamma_1(N))$, we have that $v_1, v_{1/p}, v_{1/q}, v_{1/N} \geq 0$.  It follows that $r_1, r_p, r_q, r_N \geq 0$.  Namely, if $r_1<0$, then using \eqref{spexpineq}, and assuming without loss of generality that $p<q$, we see that
\[
24v_1 < -pq |r_1| + q|r_p|+q|r_q|+q|r_N| \leq -pq |r_1| + 4q < -pq |r_1| + pq < 0,
\]
which contradicts that $v_1\geq 0$.  Similarly, if $r_p$, $r_q$, or $r_N$ is negative we get contradictions for the nonnegativity of  $v_{1/p}$, $v_{1/q}$, and $v_{1/N}$ respectively. 
  
By Corollary \ref{etathm}, we also know that 
\begin{equation}\label{firsteqn}
r_1+r_p+5r_q+5r_N \equiv 0 \pmod{24},
\end{equation}
if $(p,q) \equiv (1,5) \pmod{24},$ and
\begin{equation}\label{secondeqn}
r_1+5r_p+5r_q+r_N \equiv 0 \pmod{24}
\end{equation}
if $(p,q) \equiv (5,5) \pmod{24}$.  
  
However, neither \eqref{firsteqn} nor \eqref{secondeqn} have nonnegative integer solutions which satisfy (\ref{spexpineq}), which is a contradiction.  Thus we have shown that there are no eta-quotients in $M_2(\Gamma_1(pq))$ for $(p,q) \equiv (1,5), (5,1), (5,5) \pmod{24}$ and $p,q > 5$.


We now prove the converse by construction.  Namely, we need to show that if $h\mid k$ and it is not the case that $(p, q) \!\!\! \mod{24} \in \{(1, 5),(5,1),(5, 5)\}$, $p, q \ne 5$, and $k = 2$, then there does exist an eta-quotient in $M_k(\Gamma_1(pq))$.

We first note that setting $t=2$ in Lemma \ref{sfexistence} guarantees the existence of an eta-quotient in $M_k(\Gamma_1(pq))$ for distinct primes $p, q \geq 5$, when $k$ is an even integer divisible by $h $ such that $k(p+1)(q+1)/12$ is divisible by $4$.  Since $4$ divides $k(p+1)(q+1)/12$ whenever $4$ divides any one of $p+1$, $q+1$, or $k$, it suffices to consider only the cases of $p, q,$ and $k$ when $p+1 \equiv q+1 \equiv k \equiv 2 \pmod{4}$.  Consider the possible residues for $(p,q)$ modulo $24$, ordering so that the residue of $p$ is no larger than that of $q$.  We may immediately disregard the cases $(1,1)$, $(1,17),$ and $(17,17),$ since in each $4 \mid h$, and thus since $h \mid k$ they are covered by Lemma \ref{sfexistence}.  This leaves the cases $(1,5)$, $(1,13)$, $(5,5)$, $(5,13)$, $(5,17)$, $(13,13)$, and $(13,17)$.  It suffices to show there exists an eta-quotient in $M_h(\Gamma_1(N))$, since if $f(\tau)\in M_h(\Gamma_1(p))$, then $f^{k/h}(\tau) \in M_k(\Gamma_1(p))$.  The following table gives such eta-quotients for the cases $(1,13)$, $(5,13)$, $(5,17)$, $(13,13)$, and $(13,17)$.
 \begin{center}
 \begin{tabular}{|c|c|c|}
  \hline
  Case & $h$ & eta-quotient \\
  \hline
  $(1,13)$ & $6$ & $\eta^{11}(\tau)\eta(q\tau)$ \Tstrut\Bstrut\\
  \hline
  $(5,13)$ & $2$ & $\eta(p\tau)\eta^2(q\tau)\eta(N\tau)$ \Tstrut\Bstrut\\
  \hline
  $(5,17)$ & $2$ & $\eta(p\tau)\eta(q\tau)\eta^2(N\tau)$ \Tstrut\Bstrut\\
  \hline
  $(13,13)$ & $6$ & $\eta(q\tau)\eta^{11}(N\tau)$ \Tstrut\Bstrut\\
  \hline
  $(13,17)$ & $2$ & $\eta^2(p\tau)\eta(q\tau)\eta(N\tau)$ \Tstrut\Bstrut\\
  \hline
  \end{tabular}
 \end{center}
 This leaves the cases $(1,5)$ and $(5,5)$, both of which have $h=2$.  If either $p$ or $q$ is 5, then $M_2(\Gamma(5)) \subset M_2(\Gamma(N))$, and so by Theorem \ref{correctThm} there will exist an eta-quotient in $M_2(\Gamma(N))$.  We thus assume that neither $p$ nor $q$ is equal to $5$. As every even integer $k\geq4$ can be written as a linear combination of $4$ and $6$, it suffices to show the existence of an eta-quotient in both $M_4(\Gamma_1(N))$ and $M_6(\Gamma_1(N))$.  By Lemma \ref{sfexistence}, we have the existence of an eta-quotient in $M_4(\Gamma_1(N))$.  Moreover, one can check using Theorem \ref{1.64} that $\eta^3(\tau)\eta^9(N\tau)\in M_6(\Gamma_1(N))$ when $(p,q) \equiv (1,5) \pmod{24}$, and  $\eta^3(q\tau)\eta^9(N\tau)\in M_6(\Gamma_1(N))$ when $(p,q) \equiv (5,5) \pmod{24}$.

\end{proof}

\section{Elliptic Curves and Eta-Quotients}\label{ellcurves}

In this section, we prove Theorems \ref{N=35} and \ref{N=55}, and conclude by describing the method we used to find these examples. 

\begin{proof}[Proof of Theorem \ref{N=35}]
First using dimension formulas  (Theorems 3.5.1 and 3.1.1 in \cite{DS} for example), we calculate that $\dim_\C S_2(\G_0(35)) =3$.  By Theorems \ref{1.64} and \ref{1.65}, we see that the following three eta-quotients are members of $S_2(\G_0(35))$,
\begin{align*}
g_1(\tau) &:= \eta(\tau)\eta(5\tau)\eta(7\tau)\eta(35\tau), \\
g_2(\tau) &:= \eta(\tau)^2\eta(35\tau)^2, \\
g_3(\tau) &:= \eta(5\tau)^2\eta(7\tau)^2.
\end{align*}
The $q$-expansions of $g_1,g_2,g_3$ begin with
\begin{align*}
g_1(\tau) &= q^2 - q^3 - q^4 + q^8 + q^9 +  O(q^{10})  \\
g_2(\tau) &= q^3 - 2q^4 - q^5 + 2q^6 + q^7 + 2q^8 - 2q^9 + O(q^{10})   \\
g_3(\tau) &= q - 2q^6 - 2q^8  + O(q^{10}).  
\end{align*}
Since each $q$-expansion starts with a different power of $q$, we can quickly determine that $g_1,g_2,g_3$ are linearly independent.  Thus, they form a basis of  $S_2(\G_0(35))$, and by Theorem \ref{modularityTheorem}, any elliptic curve of conductor $35$ must be a linear combination of $g_1$, $g_2$, and $g_3$.  However, one can see for example from the L-functions and Modular Forms Database (LMFDB) \cite{lmfdb} that there is only one isogeny class of elliptic curves of conductor $35$ \cite[\href{http://www.lmfdb.org/EllipticCurve/Q/35/a/}{Elliptic Curve Isogeny Class 35.a}]{lmfdb}, and thus only one attached modular form, $f(\tau)\in S_2(\G_0(35))$ \cite[\href{http://www.lmfdb.org/ModularForm/GL2/Q/holomorphic/35/2/1/a/}{Modular Form 35.2.1.a}]{lmfdb}.  The $q$-expansion of $f$ begins with
\[
f(\tau) = q + q^3 - 2q^4 - q^5 + q^7 - 2q^9  + O(q^{10})
\]
and since the bound in Theorem \ref{sturmBound} is $8$ in this case, we see by Remark \ref{SturmRmk} that $f(\tau) =g_2(\tau) + g_3(\tau)$, as desired. 
\end{proof}

We now turn to the proof of Theorem \ref{N=55}, which requires more finesse.  

\begin{proof}[Proof of Theorem \ref{N=55}]
As with conductor $35$, there is only one isogeny class of elliptic curves of conductor $55$ \cite[\href{http://www.lmfdb.org/EllipticCurve/Q/55/a/}{Elliptic Curve Isogeny Class 55.a}]{lmfdb}, and thus only one attached modular form, $f(\tau)\in S_2(\G_0(55))$ \cite[\href{http://www.lmfdb.org/ModularForm/GL2/Q/holomorphic/55/2/1/a/}{Modular Form 55.2.1.a}]{lmfdb}.  The $q$-expansion of $f$ begins with
\[
f(\tau) = q + q^2 - q^4 + q^5 - 3q^8 - 3q^9 + q^{10} - q^{11} + 2q^{13} + O(q^{15}).
\]
The bound in Theorem \ref{sturmBound} is $12$ in this case, so we see by Remark \ref{SturmRmk}  that modular forms in $S_2(\G_0(55))$ are determined by their Fourier coefficients up to $q^{13}$.  

Using dimension formulas, we calculate that $\dim_\C S_2(\G_0(55)) =5$.  However, there are only three linearly independent eta-quotients in $S_2(\G_0(55))$ given by
\begin{align*}
g_1(\tau) &:= \eta(\tau)\eta(5\tau)\eta(11\tau)\eta(55\tau), \\
g_2(\tau) &:= \eta(\tau)^2\eta(11\tau)^2, \\
g_3(\tau) &:= \eta(5\tau)^2\eta(55\tau)^2,
\end{align*}
with $q$-expansions beginning with
 \begin{align*}
g_1(\tau) &= q^3 - q^4 - q^5 + q^9 + 2q^{10} - 2q^{13} + O(q^{15}), \\
g_2(\tau) &= q - 2q^2 - q^3 + 2q^4 + q^5 + 2q^6 - 2q^7 - 2q^9 - 2q^{10} + q^{11} - 2q^{12} + 4q^{13} + 4q^{14} + O(q^{15}), \\
g_3(\tau) &= q^5 - 2q^{10} + O(q^{15}).
\end{align*}

We thus use a method originating in work of Rouse and Webb \cite{RW}, which is to multiply $f(\tau)$ by an eta-quotient $a(\tau)$ in order to push the product $f(\tau)a(\tau)$ into a higher weight space that is generated by eta-quotients.  Then $f$ can be written as a linear combination of weakly holomorphic eta-quotients in $M_2^!(\G_0(55))$. \rmhs   

Consider the eta-quotient 
\[
a(\tau) = \eta(\tau)^3\eta(5\tau)^3\eta(11\tau)^3\eta(55\tau)^3.
\]
We see that $a(\tau)\in S_6(\G_0(55))$ by Theorems \ref{1.64} and \ref{1.65} , and so $a(\tau)f(\tau) \in S_8(\G_0(55))$.  Since the bound from Theorem \ref{sturmBound} is $48$ in this case, modular forms in $S_8(\G_0(55))$ are determined by their Fourier coefficients up to $q^{49}$.  We see that the $q$-expansion for $a(\tau)$ begins with
\begin{align*}
a(\tau) & = q^9 - 3 q^{10} + 5q^{12}  - 3q^{14} + 2q^{15} - 15 q^{17} + 9q^{19} + 18q^{20} + 9q^{21} - 15 q^{23}  - 33 q^{24} - 6 q^{25} \\ & - 6 q^{26} + 25 q^{27} + 45 q^{28} + 33 q^{29} - 49 q^{30} - 63 q^{31} + 45 q^{34} + 60 q^{35} + 45 q^{36} - 15 q^{37} - 75 q^{38} \\ & - 62 q^{39} - 78 q^{40}  + 66 q^{41} + 15 q^{42} - 15 q^{43} + 21 q^{45} + 117 q^{46} - 15 q^{47} + 55 q^{48} - 63 q^{49} + O(q^{50}),
\end{align*}
whereas the expansion for $f(\tau)$ begins with
\begin{align*}
f(\tau) & = q + q^2 - q^4 + q^5 - 3q^8 - 3q^9 + q^{10} - q^{11} + 2q^{13} - q^{16} + 6q^{17} - 3q^{18} - 4q^{19} - q^{20} \\ & - q^{22} + 4q^{23} + q^{25} + 2q^{26} + 6q^{29} - 8q^{31} + 5q^{32} + 6q^{34} + 3q^{36} - 2q^{37} - 4q^{38} - 3q^{40} \\ & + 2q^{41} + 4q^{43} + q^{44} - 3 q^{45} + 4 q^{46} - 12 q^{47} - 7 q^{49} +  O(q^{50}).
\end{align*}
Thus we see that the $q$-expansion for the product $f(\tau)a(\tau)$ starts with
\begin{align*}
f(\tau)a(\tau)  & = q^{10} - 2q^{11} - 3q^{12} + 4q^{13} + 9q^{14} -6q^{15} -6q^{16} + 4q^{17} -6q^{18} -10q^{19} -8q^{20} \\ & + 30q^{21} + 28q^{22} -8q^{23} -33q^{24} + 20q^{25} + 22q^{26} -66q^{27} -25q^{28} + 12q^{29} + 21q^{30}  \\ &  -52q^{31} + 44q^{32} + 96q^{33} + 19q^{34} + 82q^{35} -105q^{36} -86q^{37} + 29q^{38} + 108q^{39} -148q^{40}  \\ &  -106q^{41} + 159q^{42} -260q^{43} -136q^{44}  + 36q^{45} + 261q^{46} -22q^{47} + 309q^{48} + 430q^{49} + O(q^{50})
\end{align*}

Moreover, we calculate that $\dim_\C S_8(\G_0(55)) = 40$, and compute the following basis of $S_8(\G_0(55))$ consisting only of eta-quotients $\{g_1, \ldots, g_{40} \}$,  which are given in Table \ref{etaBasis}.

\begin{table}
\caption{Eta-quotient basis of $S_8(\G_0(55))$}
\label{etaBasis}
\bgroup
\def\arraystretch{1.1}
\begin{tabular}{|c|c|}
\hline
$i$ & $g_i(\tau)$ \\ \hline 
1 & $  \eta(\tau)^4\eta(5\tau)^4\eta(11\tau)^4\eta(55\tau)^4 $ \\ \hline
2 & $  \eta(\tau)^5\eta(5\tau)^3\eta(11\tau)^5\eta(55\tau)^3 $ \\ \hline
3 & $  \eta(\tau)^3\eta(5\tau)^5\eta(11\tau)^3\eta(55\tau)^5 $ \\ \hline
4 & $  \eta(\tau)^2\eta(5\tau)^6\eta(11\tau)^2\eta(55\tau)^6 $ \\ \hline
5 & $  \eta(\tau)^6\eta(5\tau)^2\eta(11\tau)^6\eta(55\tau)^2 $ \\ \hline
6 & $  \eta(\tau)^7\eta(5\tau)\eta(11\tau)\eta(55\tau)^7 $ \\ \hline
7 & $  \eta(\tau)^6\eta(5\tau)^6\eta(11\tau)^2\eta(55\tau)^2 $ \\ \hline
8 & $  \eta(\tau)^7\eta(5\tau)\eta(11\tau)^7\eta(55\tau) $ \\ \hline
9 & $  \eta(\tau)\eta(5\tau)^7\eta(11\tau)\eta(55\tau)^7 $ \\ \hline
10 & $  \eta(\tau)^8\eta(11\tau)^2\eta(55\tau)^6 $ \\ \hline
11 & $  \eta(\tau)^7\eta(5\tau)^5\eta(11\tau)^3\eta(55\tau) $ \\ \hline
12 & $  \eta(\tau)\eta(5\tau)^3\eta(11\tau)^5\eta(55\tau)^7 $ \\ \hline
13 & $  \eta(\tau)^5\eta(5\tau)^7\eta(11\tau)\eta(55\tau)^3 $ \\ \hline
14 & $  \eta(5\tau)^8\eta(55\tau)^8 $ \\ \hline
15 & $  \eta(\tau)^8\eta(11\tau)^8 $ \\ \hline
16 & $  \eta(\tau)^9\eta(5\tau)^{-1}\eta(11\tau)^3\eta(55\tau)^5 $ \\ \hline
17 & $  \eta(\tau)^5\eta(5\tau)^3\eta(11\tau)^{-1}\eta(55\tau)^9 $ \\ \hline
18 & $  \eta(\tau)^{-1}\eta(5\tau)^9\eta(11\tau)^5\eta(55\tau)^3 $ \\ \hline
19 & $  \eta(\tau)^3\eta(5\tau)^5\eta(11\tau)^9\eta(55\tau)^{-1} $ \\ \hline
20 & $  \eta(\tau)^3\eta(5\tau)^9\eta(11\tau)^5\eta(55\tau)^{-1} $ \\ \hline
21 & $  \eta(\tau)^5\eta(5\tau)^{-1}\eta(11\tau)^3\eta(55\tau)^9 $ \\ \hline
22 & $  \eta(\tau)^9\eta(5\tau)^3\eta(11\tau)^5\eta(55\tau)^{-1} $ \\ \hline
23 & $  \eta(\tau)^{-1}\eta(5\tau)^5\eta(11\tau)^3\eta(55\tau)^9 $ \\ \hline
24 & $  \eta(\tau)^4\eta(11\tau)^2\eta(55\tau)^{10} $ \\ \hline
25 & $  \eta(\tau)^3\eta(5\tau)\eta(11\tau)\eta(55\tau)^{11} $ \\ \hline
26 & $  \eta(\tau)^2\eta(5\tau)^2\eta(55\tau)^{12} $ \\ \hline
27 & $  \eta(\tau)\eta(5\tau)^3\eta(11\tau)^{-1}\eta(55\tau)^{13} $ \\ \hline
28 & $  \eta(\tau)^6\eta(5\tau)^{-2}\eta(11\tau)^{-2}\eta(55\tau)^{14} $ \\ \hline
29 & $  \eta(5\tau)^4\eta(11\tau)^{-2}\eta(55\tau)^{14} $ \\ \hline
30 & $  \eta(11\tau)^2\eta(55\tau)^{14} $ \\ \hline
31 & $  \eta(\tau)\eta(5\tau)^{-1}\eta(11\tau)^{-3}\eta(55\tau)^{19} $ \\ \hline
32 & $  \eta(\tau)\eta(5\tau)^7\eta(11\tau)^7\eta(55\tau) $ \\ \hline
33 & $  \eta(\tau)^2\eta(5\tau)^2\eta(11\tau)^6\eta(55\tau)^6 $ \\ \hline
34 & $  \eta(\tau)^3\eta(5\tau)\eta(11\tau)^7\eta(55\tau)^5 $ \\ \hline
35 & $  \eta(\tau)^2\eta(5\tau)^6\eta(11\tau)^8 $ \\ \hline
36 & $  \eta(\tau)^6\eta(5\tau)^2\eta(55\tau)^8 $ \\ \hline
37 & $  \eta(5\tau)^8\eta(11\tau)^6\eta(55\tau)^2 $ \\ \hline
38 & $  \eta(\tau)^4\eta(11\tau)^8\eta(55\tau)^4 $ \\ \hline
39 & $  \eta(\tau)^8\eta(5\tau)^4\eta(11\tau)^4  $ \\ \hline
40 & $  \eta(5\tau)^4\eta(11\tau)^4\eta(55\tau)^8  $ \\ \hline
\end{tabular}
\egroup
\end{table}

From their $q$-expansions, we calculate that 
\[
f(\tau)a(\tau) = \sum_{i=1}^{40} c_i g_i(\tau),
\]
where the coefficients $c_i$ are given in Table \ref{coeffTab}.

\begin{table}
\caption{Coefficients $c_i$}
\label{coeffTab}
\bgroup
\def\arraystretch{1.1}
\begin{tabular}{|c|c|}
\hline
$i$ & $c_i$ \\ \hline
1&    -6008649555929309389497/819506238451459924562 \\ \hline
 2&    -502520890503551696366/409753119225729962281 \\ \hline
3 &  -18079466846617647763574/1229259357677189886843 \\ \hline
4 &  -9707713817545985330315/1639012476902919849124 \\ \hline
  5 &   -256094683592582994017/819506238451459924562 \\ \hline
  6 &  -801755327323567495694/107829768217297358495 \\ \hline
 7 &-5830988018825221370539/12292593576771898868430 \\ \hline
8 &   -779344877836568799883/2458518715354379773686 \\ \hline
9 &   8222210796731963837135/1229259357677189886843 \\ \hline
10 &-11036620216580334273019/24585187153543797736860 \\ \hline
11 &  -1229041712172689367604/6146296788385949434215 \\ \hline
12 &-111561550514099684140912/2048765596128649811405 \\ \hline
13 &    107114413382088962036/2048765596128649811405 \\ \hline
14 &    -341358820591409973660/409753119225729962281 \\ \hline
15 &   -152612595137164539575/2458518715354379773686 \\ \hline
 16 &                                              0 \\ \hline
 17 &  -7213279306787331549849/819506238451459924562 \\ \hline
 18 &     32952053389588166394/409753119225729962281 \\ \hline
 19  &   30522519027432907915/2458518715354379773686 \\ \hline
   20 &                                            0 \\ \hline
  21 & -1552252178781785948225/819506238451459924562 \\ \hline
  22 &                                             0 \\ \hline
 23  &     15742588981996697524/11074408627722431413 \\ \hline
24 &-116284353318788030011856/1229259357677189886843 \\ \hline
25 & -532560894104521109482105/1639012476902919849124 \\ \hline
26 &-1514700568262560995033025/2458518715354379773686 \\ \hline
27 & -561586338302539429500115/819506238451459924562 \\ \hline
 28 &   -149810619258256219780/409753119225729962281 \\ \hline
29 & -170766572836167986762205/819506238451459924562 \\ \hline
30 &  -81633970222789942752773/409753119225729962281 \\ \hline
 31 & 151568343772098094548465/819506238451459924562 \\ \hline
 32 &   845017548690592920502/6146296788385949434215 \\ \hline
33 &-343045582489873801897249/12292593576771898868430 \\ \hline
34 &-750234154608987483240971/24585187153543797736860 \\ \hline
35  &  -297084239198956028146/6146296788385949434215 \\ \hline
36 & -30834650074592581387973/2048765596128649811405 \\ \hline
 37  &    -9228870808743396499/107829768217297358495 \\ \hline
 38 &  -4131906244671465777242/409753119225729962281 \\ \hline
39  &   -30522519027432907915/2458518715354379773686 \\ \hline
40  & -13146446749054646251719/819506238451459924562 \\ \hline
\end{tabular}
\egroup
\end{table}

Thus, we can write
\begin{equation}\label{f-linearcomb}
f(\tau)= \sum_{i=1}^{40} c_i \frac{g_i(\tau)}{a(\tau)}, 
\end{equation}
where the simplified eta-quotients $g_i(\tau)/a(\tau)$ are given in Table \ref{etaTab}.
\end{proof}

\begin{table}
\caption{Eta-quotients $g_i(\tau)/a(\tau)$}
\label{etaTab}
\bgroup
\def\arraystretch{1.1}
\begin{tabular}{|c|c|}
\hline
$i$ & $g_i(\tau)/a(\tau)$ \\ \hline 
1& $\eta(\tau)\eta(5\tau)\eta(11\tau)\eta(55\tau)$    \\ \hline
 2& $\eta(\tau)^2\eta(11\tau)^2$    \\ \hline
3 & $\eta(5\tau)^2\eta(55\tau)^2$  \\ \hline
4 & $\eta(\tau)^{-1}\eta(5\tau)^3\eta(11\tau)^{-1}\eta(55\tau)^3$  \\ \hline
  5 & $\eta(\tau)^3\eta(5\tau)^{-1}\eta(11\tau)^3\eta(55\tau)^{-1}$   \\ \hline
  6 & $\eta(\tau)^4\eta(5\tau)^{-2}\eta(11\tau)^{-2}\eta(55\tau)^4$  \\ \hline
 7 & $\eta(\tau)^3\eta(5\tau)^3\eta(11\tau)^{-1}\eta(55\tau)^{-1}$ \\ \hline
8 & $\eta(\tau)^4\eta(5\tau)^{-2}\eta(11\tau)^4\eta(55\tau)^{-2}$   \\ \hline
9 & $\eta(\tau)^{-2}\eta(5\tau)^4\eta(11\tau)^{-2}\eta(55\tau)^4$   \\ \hline
10 & $\eta(\tau)^5\eta(5\tau)^{-3}\eta(11\tau)^{-1}\eta(55\tau)^3$ \\ \hline
11 & $\eta(\tau)^4\eta(5\tau)^2\eta(55\tau)^{-2}$  \\ \hline
12 & $\eta(\tau)^{-2}\eta(11\tau)^2\eta(55\tau)^4$ \\ \hline
13 & $\eta(\tau)^2\eta(5\tau)^4\eta(11\tau)^{-2}$   \\ \hline
14 & $\eta(\tau)^{-3}\eta(5\tau)^5\eta(11\tau)^{-3}\eta(55\tau)^5$  \\ \hline
15 & $\eta(\tau)^5\eta(5\tau)^{-3}\eta(11\tau)^5\eta(55\tau)^{-3}$  \\ \hline
 16 & $\eta(\tau)^6\eta(5\tau)^{-4}\eta(55\tau)^2$    \\ \hline
 17 & $\eta(\tau)^2\eta(11\tau)^{-4}\eta(55\tau)^6$  \\ \hline
 18 & $\eta(\tau)^{-4}\eta(5\tau)^6\eta(11\tau)^2$    \\ \hline
 19  & $\eta(5\tau)^2\eta(11\tau)^6\eta(55\tau)^{-4}$   \\ \hline
   20 & $\eta(5\tau)^6\eta(11\tau)^2\eta(55\tau)^{-4}$   \\ \hline
  21 & $\eta(\tau)^2\eta(5\tau)^{-4}\eta(55\tau)^6$  \\ \hline
  22 & $\eta(\tau)^6\eta(11\tau)^2\eta(55\tau)^{-4}$   \\ \hline
 23  & $\eta(\tau)^{-4}\eta(5\tau)^2\eta(55\tau)^6$   \\ \hline
24 & $\eta(\tau)\eta(5\tau)^{-3}\eta(11\tau)^{-1}\eta(55\tau)^7$ \\ \hline
25 & $\eta(5\tau)^{-2}\eta(11\tau)^{-2}\eta(55\tau)^8$ \\ \hline
26 & $\eta(\tau)^{-1}\eta(5\tau)^{-1}\eta(11\tau)^{-3}\eta(55\tau)^9$ \\ \hline
27 & $\eta(\tau)^{-2}\eta(11\tau)^{-4}\eta(55\tau)^{10}$ \\ \hline
 28 & $\eta(\tau)^3\eta(5\tau)^{-5}\eta(11\tau)^{-5}\eta(55\tau)^{11}$   \\ \hline
29 & $\eta(\tau)^{-3}\eta(5\tau)\eta(11\tau)^{-5}\eta(55\tau)^{11}$ \\ \hline
30 & $\eta(\tau)^{-3}\eta(5\tau)^{-3}\eta(11\tau)^{-1}\eta(55\tau)^{11}$  \\ \hline
 31 & $\eta(\tau)^{-2}\eta(5\tau)^{-4}\eta(11\tau)^{-6}\eta(55\tau)^{16}$ \\ \hline
 32 & $\eta(\tau)^{-2}\eta(5\tau)^4\eta(11\tau)^4\eta(55\tau)^{-2}$  \\ \hline
33 & $\eta(\tau)^{-1}\eta(5\tau)^{-1}\eta(11\tau)^3\eta(55\tau)^3$ \\ \hline
34 & $\eta(5\tau)^{-2}\eta(11\tau)^4\eta(55\tau)^2$  \\ \hline
35  & $\eta(\tau)^{-1}\eta(5\tau)^3\eta(11\tau)^5\eta(55\tau)^{-3}$  \\ \hline
36 & $\eta(\tau)^3\eta(5\tau)^{-1}\eta(11\tau)^{-3}\eta(55\tau)^5$  \\ \hline
 37  & $\eta(\tau)^{-3}\eta(5\tau)^5\eta(11\tau)^3\eta(55\tau)^{-1}$    \\ \hline
 38 & $\eta(\tau)\eta(5\tau)^{-3}\eta(11\tau)^5\eta(55\tau)$  \\ \hline
39  & $\eta(\tau)^5\eta(5\tau)\eta(11\tau)\eta(55\tau)^{-3}$  \\ \hline
40  & $\eta(\tau)^{-3}\eta(5\tau)\eta(11\tau)\eta(55\tau)^5$   \\ \hline
\end{tabular}
\egroup
\end{table}

We conclude by describing in an algorithmic fashion the method we used to obtain Theorems \ref{N=35} and \ref{N=55}.  We have seen that both of these theorems, once discovered, can be proved in a straightforward manner.  However, how to find results like these may not be immediately apparent.  Our approach, which is inspired by work of Pathakjee, RosnBrick, and Yoong \cite{PRY}, utilizes results from Section \ref{section4} and has the potential to generate many new examples. We note that many of the steps require the aid of mathematical software to be practical.  We used SageMath \cite{sagemath}. \\

\noindent {\bf Step 1.}  Fix a semiprime $N=pq$ satisfying the conditions in Theorem \ref{mastersemi} with $k=2$ that is the conductor of an elliptic curve $E$.  By the Modularity Theorem (Theorem \ref{modularityTheorem}) we know that $E$ has an associated modular form $f(\tau)\in S_2(\G_0(N))$. \\
\noindent {\bf Step 2.}  Compute $d_N= \dim_\C S_2(\G_0(N))$. \\
\noindent {\bf Step 3.}  Compute all partitions of $(p+1)(q+1)/6$ into exactly four parts, and construct distinct rearrangements in order to get a complete list of all possible tuples $(v_1, v_{1/p}, v_{1/q}, v_N)\in \N^4$ satisfying
\begin{equation}\label{vsum}
v_1 +  v_{1/p} + v_{1/q} + v_{1/N} = \frac{(p+1)(q+1)}{6}. 
\end{equation} 
By Theorem \ref{sfvansum}, we know that any eta-quotient in $S_2(\G_0(N))$ must have orders of vanishing $v_1, v_{1/p}, v_{1/q}, v_{1/N} \geq 1$ that satisfy \eqref{vsum}. \\
\noindent {\bf Step 4.}  For each tuple $(v_1, v_{1/p}, v_{1/q}, v_N)$ obtained in Step 3, use Theorem \ref{1.65} to construct the following system of four equations in the four unknowns $(r_1, r_p, r_q, r_N)$
\begin{align*}
  24v_1 &= Nr_1+qr_p+pr_q+r_N \\
  24v_{1/p} &= qr_1 + Nr_p + r_q + pr_N \\
  24v_{1/q} &= pr_1 + r_p + Nr_q + qr_N \\
  24v_{1/N} &= r_1 + pr_p + qr_q + Nr_N,
\end{align*}
and solve for the unique solution $(r_1, r_p, r_q, r_N)\in \Q^4$. \\
\noindent {\bf Step 5.}  For each tuple $(r_1, r_p, r_q, r_N)$ from Step 4 that has integer entries, let 
\[
g(\tau)=\eta(\tau)^{r_1}\eta(p\tau)^{r_p}\eta(q\tau)^{r_q}\eta(N\tau)^{r_N}
\]
and use Theorems \ref{1.64} and \ref{1.65} to check whether $g(\tau) \in S_2(\G_0(N))$.  List  all such $g\in S_2(\G_0(N))$.  \\
\noindent {\bf Step 6.}   Construct a maximally sized linearly independent set of eta-quotients from the list in Step 5, using linear algebra.\\
\noindent {\bf Step 7.}  If the set from Step 6 has size $d_N$, then it forms a basis of  $S_2(\G_0(N))$.  In this case, compute the Sturm Bound from Theorem \ref{sturmBound} and write $f$ as a linear combination of the basis from Step 6.  If not, go to Step 8.\\
\noindent {\bf Step 8.}  Repeat Steps 2-6 for weights $2,4,6,\ldots$ until a weight  $k$ is found such that $S_k(\G_0(N))$ has a basis of eta-quotients, and $S_{k-2}(\G_0(N))$ contains an eta-quotient $a(\tau)$.  Compute the Sturm Bound from Theorem \ref{sturmBound} and write $f(\tau)a(\tau)$ as a linear combination of the basis.  Divide through by $a(\tau)$ to write $f(\tau)$ as a linear combination of eta-quotients in $M_2^!(\G_0(N))$.

\end{document}